\documentclass[12pt]{amsart}
\pdfoutput=1
\usepackage{fullpage}
\usepackage{amssymb}
\usepackage{amstext}
\usepackage{amsmath}
\usepackage{amsthm}
\usepackage{stmaryrd}
\usepackage{verbatim}
\usepackage{booktabs}
\usepackage{array}
\usepackage{float}
\usepackage{hyperref}
\newtheorem{theorem}{Theorem}
\newtheorem{corollary}[theorem]{Corollary}
\newtheorem{proposition}[theorem]{Proposition}
\newtheorem{lemma}[theorem]{Lemma}
\theoremstyle{remark}

\newcommand{\SL}{\operatorname{SL}}
\newcommand{\D}{\operatorname{D}}
\newcommand{\spt}{\operatorname{spt}}
\newcommand{\hol}{\operatorname{hol}}
\newcommand{\ord}{\operatorname{ord}}
\begin{document}
\title{Explicit constructions for Ramanujan-type congruences}
\author{Wei Wang}
\address{Department of Mathematics, Shaoxing University, Shaoxing 312000, China}
\email{weiwang\_math@163.com}

\begin{abstract}
For an integer-valued sequence $\{a(n)\}_{n\geq 0}$ and a prime $\ell$, a Ramanujan-type congruence is a relation of the form $a(\ell n-\delta_{\ell})\equiv 0\pmod\ell$, where $\delta_{\ell}$ is a specific shift. In this paper, we present explicit constructions of modular forms to establish Ramanujan-type congruences for a broad class of generating functions, including eta-quotients, weakly holomorphic modular forms, and mock modular forms. As applications, our explicit approach provides a unified framework to not only recover known congruences but also establish new non-congruence results and explicit congruences for various combinatorial and arithmetic functions.
\end{abstract}
\keywords{Ramanujan-type congruences, Rankin-Cohen brackets, Mock modular forms, Eta-quotients, Partition}
\subjclass[2020]{11F33, 11F37, 11P83}
\thanks{ }
\maketitle
\section{Introduction}
Ramanujan found the following famous congruences for the partition function $p(n)$:
\begin{align*}
&p\left(5n+4\right)\equiv 0\pmod 5,\\
&p(7n+5)\equiv 0\pmod 7,\\
&p(11n+6)\equiv 0\pmod {11}.
\end{align*}
Since the discovery of these congruences, numerous approaches have been developed to establish them. For instance, Ramanujan proved the congruence modulo 5 by employing the following identity:
\[
\sum_{n\geq 0}p(5n+4)q^{5n}=5\frac{(q^5;q^5)_{\infty}^5}{(q;q)^6_{\infty}},
\]
where $(q;q)_{\infty}:=\prod_{n\geq 1}(1-q^n)$ denotes the $q$-Pochhammer symbol. Another key insight lies in establishing the connection between these congruences and modular forms.  For prime $5 \leq \ell\leq 23$, Ramanujan claimed that the $q$-series
\begin{equation}\label{eq_Ram_par}
(q^{\ell};q^{\ell})_{\infty}\sum_{n\geq 0}p\left(\ell n-\delta_{\ell}\right)q^n,\quad\text{where }\delta_{\ell}:=(\ell^2-1)/24,
\end{equation}
is congruent to a modular form of weight $\ell-1$ on $\SL_2(\mathbb{Z})$, see \cite[Section 15]{zbMATH01301076}.  Note that when $\ell=5,7,11$, this result immediately yields Ramanujan's congruences, as the corresponding spaces of cusp forms are trivial.
Chua \cite{zbMATH02051572} proved this for any prime $\ell\geq 5$, and Ahlgren and Boylan \cite{zbMATH02001026} extended this result to arbitrary powers of primes. A crucial point is that the generating function of the partition function 
\[
\sum_{n\geq 0}p(n)q^{n-1/24}=q^{-1/24}\frac{1}{(q;q)_{\infty}}=\frac{1}{\eta(\tau)},\quad q=e^{2\pi i \tau},
\]
is a weakly holomorphic modular form. The starting point of their proofs is to consider the weakly holomorphic modular form 
\[
\frac{\eta^{\ell}(\ell \tau)}{\eta(\tau)}\in M^{!}_{\ell-1}(\ell),
\]
which is congruent to a power of the $\Delta$ function modulo $\ell$:
\begin{equation}\label{eq_con_par}
\frac{\eta^{\ell}(\ell \tau)}{\eta(\tau)}\equiv \Delta(\tau)^{\frac{\ell^2-1}{24}} \pmod {\ell}.
\end{equation}
Applying the $U_{\ell}$ operator defined by
\[
\left(\sum a(n)q^n\right)\big|U_{\ell}:=\sum a(\ell n)q^n
\]
to both sides of \eqref{eq_con_par} and utilizing the fundamental properties of the $U_{\ell}$ operator:
\begin{equation}\label{eq_propertyUl}
\left(\sum a(n)q^{n}\sum b(n)q^{\ell n}\right)\big|U_{\ell}:=\left(\sum a(n)q^n\right)\big|U_{\ell}\sum b(n)q^{n},
\end{equation}
yields
\[
(q^{\ell},q^{\ell})_{\infty}\sum_{n\geq 0}p(\ell n-\delta_{\ell})q^n\equiv \left(\Delta(\tau)^{\frac{\ell^2-1}{24}} \right)\big|U_{\ell}\pmod\ell.
\]
By employing Serre's filtration \cite[Proposition 3.3]{zbMATH06011365} or the trace operator \cite{zbMATH02051572}, it can be shown that the right-hand side of the above equation is congruent to a modular form of weight $\ell-1$, and the proof is complete.

Whether relying on Serre's filtration or the trace operator, the aforementioned approaches provide merely an existence proof and fail to explicitly construct the specific modular form of weight $\ell-1$. In this paper, we constructively provide an explicit expression for this modular form by utilizing a specific bracket operator on the space of modular forms, namely the Rankin-Cohen bracket. For functions $f$ and $g$ defined on the upper half-plane and $a,b\in \frac{1}{2}\mathbb{Z}$, it is defined as
\begin{equation}\label{eq_RC_def}
[f,g]_{v}:=\sum_{0\leq j\leq v} (-1)^j \binom{a+v-1}{v-j} \binom{b+v-1}{j}\D^jf\D^{v-j}g.
\end{equation}
Recently, Bringmann, Craig and Ono \cite{zbMATH08123433} gave a closed form expression of the generating function \eqref{eq_Ram_par} using the Hecke traces of $\ell$-ramified values of special Dirichlet series for weight $\ell-1$ cusp forms. Their approach, in essence, utilizes the Rankin-Cohen bracket.

Throughout this paper, for any sequence $a(n)$,  we adopt the convention that $a(n)=0$ for $n<0$. We classify our results into three distinct types. Type I addresses eta-quotients of weight $-1/2$, with the ordinary partition function serving as the most fundamental example.
\begin{theorem}[\textbf{Type I}]\label{thm_main1}
Given an eta-quotient
\[
f(\tau)=\sum_{n\geq 0}a(n)q^{n-\frac{t}{w}}=\prod_{\delta\mid N}\eta(\delta \tau)^{-r_{\delta}},\quad \frac{t}{w}:=\frac{\sum_{\delta\mid N}\delta r_{\delta}}{24}.
\] 
We make the following assumptions:
\begin{enumerate}
\item There is no common factor among the $\delta$ for which $r_{\delta}$ is non-zero, that is, we assume that $f$ is a primitive eta-quotient.

\item $g:=1/f$ is a holomorphic modular form of weight $1/2$, level $N$, that is, for any $\gamma=\begin{pmatrix}
a& b\\
c& d
\end{pmatrix}\in \Gamma_0(N)$, there is a  multiplier $\nu$ such that
\[
g\left(\frac{a\tau+b}{c\tau+d}\right)=\nu(\gamma)(c\tau+d)^{1/2}g(\tau).
\]
\item The Fourier expansion of $g(\tau)$ is given by $\sum_{m\in\mathbb{Z}}b(m)q^{\frac{m^2}{w}}$, and satisfies $b(\ell m)=\chi(\ell)b(m)$ for every $m$.
\end{enumerate}
For every prime $\ell\geq 5$, we have
\[
\chi(\ell)\cdot\prod_{\delta\mid N}(q^{\ell\delta};q^{\ell\delta})^{r_{\delta}}_{\infty}\sum_{n\geq 0}a(\ell n-\tfrac{t}{w}(\ell^2-1))q^n\equiv 1-\left(\frac{-w}{\ell}\right)\frac{8}{3}\cdot[f(\tau),g(\tau)]_{(\ell-1)/2}\big|U_{\ell}\pmod{\ell} .
\]
and $[f(\tau),g(\tau)]_{(\ell-1)/2}$ is a holomorphic modular form in $M_{\ell-1}(\Gamma_0(N))$.
\end{theorem}

There are only finitely many pairs $(f,g)$ satisfying these conditions. Although Condition (1) is not explicitly utilized in the proof, it is included to circumvent the discussion of redundant generating functions. According to Mersmann's classification \cite{mersmann1991holomorphe} of primitive eta-quotients of weight $1/2$ (see also \cite[p.30]{zbMATH05808162}), there are only a few choices of $g$ satisfying conditions (1) and (2):
\begin{align*}
& \eta(\tau), \quad \frac{\eta(\tau)^2}{\eta(2\tau)}, \quad \frac{\eta(2\tau)^2}{\eta(\tau)}, \quad \frac{\eta(\tau)\eta(4\tau)}{\eta(2\tau)}, \quad \frac{\eta(2\tau)^3}{\eta(\tau)\eta(4\tau)}, \quad \frac{\eta(2\tau)^5}{\eta(\tau)^2\eta(4\tau)^2}, \\[1em]
& \frac{\eta(\tau)^2\eta(6\tau)}{\eta(2\tau)\eta(3\tau)}, \quad \frac{\eta(2\tau)^2\eta(3\tau)}{\eta(\tau)\eta(6\tau)}, \quad \frac{\eta(2\tau)\eta(3\tau)^2}{\eta(\tau)\eta(6\tau)}, \quad \frac{\eta(\tau)\eta(6\tau)^2}{\eta(2\tau)\eta(3\tau)}, \\[1em]
& \frac{\eta(\tau)\eta(4\tau)\eta(6\tau)^2}{\eta(2\tau)\eta(3\tau)\eta(12\tau)}, \quad \frac{\eta(2\tau)^2\eta(3\tau)\eta(12\tau)}{\eta(\tau)\eta(4\tau)\eta(6\tau)}, \quad \frac{\eta(2\tau)^5\eta(3\tau)\eta(12\tau)}{\eta(\tau)^2\eta(4\tau)^2\eta(6\tau)^2}, \\[1em]
& \frac{\eta(\tau)\eta(4\tau)\eta(6\tau)^5}{\eta(2\tau)^2\eta(3\tau)^2\eta(12\tau)^2} .
\end{align*}
All these eta-quotients satisfy the condition (3), see the appendix. For an explanation of the $q$-series in the appendix, we refer the reader to \cite[Chapter 8]{kohler2011eta}.

Our second result focuses on weakly holomorphic modular forms of weight $3/2$.
\begin{theorem}[\textbf{Type II}]\label{thm_main2}
Let $f(\tau)$ be a weakly holomorphic modular form of weight $3/2$ with Fourier expansion
\[
f(\tau)=\sum_{n\geq 0}a(n)q^{n-\frac{t}{w}},~a(n)\in\mathbb{Z}_{(\ell)}.
\] 
We make the following assumptions:
\begin{enumerate}
\item There is an eta-quotient $g(\tau)=\prod_{\delta\mid N}\eta(\delta\tau)^{r_{\delta}}$ of weight $1/2$ such that $fg$ is a holomorphic modular form in $M_2(\Gamma_0(N))$. 

\item The function $g(\tau)$ satisfies the conditions (2) in Theorem \ref{thm_main1}.

\item Let $s$ be an integer coprime to $w$. The Fourier expansion of $g(\tau)$ is given by $\sum_{m\in\mathbb{Z}}b(m)q^{\frac{sm^2}{w}}$, and satisfies $b(\ell m)=\chi(\ell)b(m)$ for every $m$.
\end{enumerate}
For every prime $\ell\geq 3$ with $\ell\nmid sw$, we have
\[
\chi(\ell)\cdot\prod_{\delta\mid N}(q^{\ell\delta};q^{\ell\delta})^{r_{\delta}}_{\infty}\sum_{n\geq 0}a(\ell n-\tfrac{t}{w}(\ell^2-1))q^n\equiv \left(fg-\left(\frac{-sw}{\ell}\right)\cdot[f(\tau),g(\tau)]_{(\ell-1)/2}\right)\big|U_{\ell}\pmod{\ell} .
\]
and $[f(\tau),g(\tau)]_{(\ell-1)/2}$ is a holomorphic modular form in $M_{\ell+1}(\Gamma_0(N))$. Moreover, if $fg$ is a cusp form, then $[f(\tau),g(\tau)]_{(\ell-1)/2}$ is also a cusp form in $S_{\ell+1}(\Gamma_0(N))$.
\end{theorem}

Type III deals with mock modular forms of weight $3/2$. Mock modular forms are a  significant class of $q$-series; however, they do not possess modular transformation properties on their own. Adding a suitable non-holomorphic part transforms them into harmonic weak Maass forms that satisfy modular transformations 
\[
F(\gamma\tau)=\nu(\gamma)\,(c\tau+d)^k F(\tau)
\text{ for every }
\gamma=\begin{pmatrix}a&b\\ c&d\end{pmatrix}\in\Gamma.
\] 
The harmonic weak Maass form $F(\tau)$ has the Fourier expansion of the form
\[
F(\tau)=f^{+}(\tau)+f^-(\tau),
\]
where
\[
f^+(\tau)=\sum_{n\gg -\infty} c_f^+(n)\,q^{n+\alpha},
\]
is the generating function we are interested in, and
\[
f^-(\tau)
=
(4\pi)^{1-k}
\left(
\frac{c_f^-(0)}{k-1}y^{1-k}
+
\sum_{m<0}
c_f^-(m)\,|m+\alpha|^{k-1}
\Gamma\!\left(1-k,4\pi|m+\alpha|y\right)q^{m+\alpha}
\right).
\]
is its non-holomorphic part. The function
\[
\xi_kF(\tau)
:=\overline{c_f^-(0)}+\sum_{n<0}\overline{c_f^-(n)}q^{-n-\alpha}.
\]
is called the shadow of $f^+(\tau)$. If $F(\tau)$ is a harmonic weak Maass form of weight $2-k$ and multiplier $\nu$, by the Bruinier–Funke theorem \cite[Proposition 3.2]{zbMATH02135351}, its shadow $(\xi_{2-k}F)(\tau)$ is a weakly holomorphic modular form of weight $k$ and multiplier $\overline{\nu}$.
\begin{theorem}[\textbf{Type III}]\label{thm_main3}
Let $f(\tau)$ be a mock modular form of weight $3/2$ with Fourier expansion
\[
f(\tau)=\sum_{n\geq 0}a(n)q^{n-\frac{t}{w}},~a(n)\in\mathbb{Z}_{(\ell)},
\]
which means that $f(\tau)$ is the holomorphic part of a  harmonic weak Maass form $F(\tau)$ of level $N$.
We make the following assumptions:
\begin{enumerate}
\item The shadow of $F(\tau)$ is a holomorphic modular form of weight $1/2$ whose Fourier coefficients are $\ell$-integral. By the Serre-Stark basis theorem, it is a linear combination of unary theta functions $\sum_{n\in\mathbb{Z}}\psi(n)q^{un^2/w}$ for some positive integers $u$ and a character (or function) $\psi$. We assume that this linear combination is in $\mathbb{Z}_{(\ell)}$.

\item There is a function $g(\tau)$ such that $Fg$ behaves like a modular form of weight $2$, that is, for any $\gamma=\begin{pmatrix}
a& b\\
c& d
\end{pmatrix}\in \Gamma_0(N)$, 
\[
(Fg)\left(\frac{a\tau+b}{c\tau+d}\right)=(c\tau+d)^{2}(Fg)(\tau),
\]
and $Fg$ has no principal parts at any cusp of $\Gamma_0(N)$.

\item The function $g(\tau)$ is an eta-quotient $g(\tau)=\prod_{\delta\mid N}\eta(\delta\tau)^{r_{\delta}}$ and satisfies the condition (2) in Theorem \ref{thm_main1}.

\item Let $s$ be an integer coprime to $w$.  The Fourier expansion of $g(\tau)$ is given by $\sum_{m\in\mathbb{Z}}b(m)q^{\frac{sm^2}{w}}$, and satisfies $b(\ell m)=\chi(\ell)b(m)$ for every $m$.  For any  $u$ in condition (1),  we need $\left(\frac{su}{\ell}\right)=1$.

\end{enumerate}
For every prime $\ell\geq 3$ with $\ell\nmid sw$, the $q$-series
\[
\prod_{\delta\mid N}(q^{\ell\delta};q^{\ell\delta})^{r_{\delta}}_{\infty}\sum_{n\geq 0}a(\ell n-\tfrac{t}{w}(\ell^2-1))q^n
\]
is congruent modulo $\ell$ to a modular form in $M_{\ell+1}(\Gamma_0(N))$.
\end{theorem}
By the Bruinier–Funke theorem, the shadow $(\xi_{1/2}F)(\tau)\in M_{1/2}^{!}(\Gamma_0(N),\overline{\nu})$. In many cases, we choose $g$ to be the shadow of $f$, which naturally satisfies the condition (2).  Although the results in Type III appear to be merely an existence argument, according to its proof, we have actually constructed an explicit modular form. Furthermore, the expansions of this modular form at different cusps are computable, and its Eisenstein part can be explicitly obtained.

Given a congruence subgroup $\Gamma_0(N)$, the space of modular forms has a decomposition: 
\[
M_{k}(\Gamma_0(N))=\mathcal{E}_k(\Gamma_0(N))\oplus S_k(\Gamma_0(N))
\]
where $\mathcal{E}_k(\Gamma_0(N))$ and $S_k(\Gamma_0(N))$ denotes the $\mathbb{C}$-space of Eisenstein series and  cusp forms, respectively. Since $E_{\ell-1}\equiv 1\pmod{\ell}$,  our theorems imply that the $q$-series
\[
\prod_{\delta\mid N}(q^{\ell\delta};q^{\ell\delta})^{r_{\delta}}_{\infty}\sum_{n\geq 0}a(\ell n-\tfrac{t}{w}(\ell^2-1))q^n
\]
is congruent modulo $\ell$ to an explicit function. Our explicit construction has an advantage: because the values of the modular form $[f(\tau),g(\tau)]_{(\ell-1)/2}$ (in Type I and Type II) at each cusp can be easily computed, we can obtain its Eisenstein part explicitly. This is very helpful for determining its congruence in many cases. 

In the case of Type I, we illustrate the power of this explicit construction using the partition function as an example. Suppose that for a given prime $\ell\geq 5$, there exists a Ramanujan-type congruence $\sum_{n\geq 0}p(\ell n-\delta_{\ell})q^n\equiv 0\pmod{\ell}$. This implies that, modulo $\ell$, the cusp form
\[
f_{\ell}:=E_{\ell-1}-\left(\frac{-w}{\ell}\right)\frac{8}{3}\cdot [1/\eta,\eta]_{(\ell-1)/2}
\] 
lies in the kernel of $U_{\ell}$. By a result of Jochnowitz \cite[Corollary 7.7]{zbMATH03851200}, this kernel is trivial, which forces $f_{\ell}\equiv 0\pmod{\ell}$. However, since our construction is explicit, by computing its Fourier coefficients, we find that this can only occur when $\ell\in\{5,7,11\}$. Thus, we provide a new proof of Ahlgren and Boylan's theorem \cite{zbMATH02001026}:  
\begin{corollary}\label{coro_inconforpar}
For any prime $\ell\geq 13$,
\[
\sum_{n\geq 0}p(\ell n-\delta_{\ell})q^n\not\equiv 0\pmod{\ell}.
\]
\end{corollary}
We handle all the eta-quotients in Mersmann’s list and prove the following non-congruence result:
\begin{corollary}\label{coro_inconforTypeI}
For all eta-quotients in Type I, excluding $1/\eta(\tau)$ and $\eta(\tau)\eta(4\tau)/\eta(2\tau)^3$, there exist no Ramanujan-type congruences modulo $\ell$ for any prime $\ell\geq 5$. More precisely, if $t=0$,
\[
\sum_{n\geq 0}a(\ell n)q^n\not\equiv 1\pmod{\ell},
\]
and if $t\neq 0$,
\[
\sum_{n\geq 0}a(\ell n-\tfrac{t}{w}(\ell^2-1))q^n\not\equiv 0\pmod{\ell}.
\]
\end{corollary}
We note that this result is not covered by the non-congruence results established in previous literatures (c.f. \cite{zbMATH06353661, zbMATH06789145}). Our results, combined with those in \cite[Proposition 1]{zbMATH07343720}, assert that there do not exist any of the following Ramanujan congruences:
\[
a(\ell n+\beta)\equiv 0\pmod\ell, \text{ for }\beta\in\{0,1,\ldots,\ell-1\}.
\]
It is not surprising that the function $\eta(\tau)\eta(4\tau)/\eta(2\tau)^3$ stands as an exception, as it is the generating function for the partition function, by the identity:
\[
\frac{(q;q)_{\infty}(q^4;q^4)_{\infty}}{(q^2;q^2)^3_{\infty}}=\frac{1}{(-q;-q)_{\infty}}.
\]

In Type II and Type III, the weight of the space of modular forms in which the congruence occurs is $\ell+1$. Under this specific weight, and restricting to certain conditions, we can prove that $U_{\ell}^2=\mathrm{Id}$ modulo $\ell$. Consequently, we establish the following congruences.
\begin{corollary}\label{coro_type2}
If the pair $(f,g)$ belongs to Type II or Type III, and the genus of $X_0(N)$ is zero with $\ell\nmid N$ and $N\neq 25$, then for every $n\geq 0$,
\[
\chi(\ell)a\left(\ell^2n-\tfrac{t}{w}(\ell^2-1)\right)\equiv  \left(1-\left(\frac{s(t-wn)}{\ell}\right)\right)a(n)\pmod\ell.
\]
Moreover, if $a(n)\neq 0$ for any sufficiently large $n$, let $S$ be the set of primes $\ell$ such that the following congruence holds:
\[
\sum_{n\geq 0}a(\ell n-\tfrac{t}{w}(\ell^2-1))q^n \equiv 0 \pmod \ell,
\]
then the natural density of $S$ is $0$.
\end{corollary}

One of the most important example of Type III comes from the smallest part of partitions $\spt(n)$. Andrews \cite[Theorem 4]{zbMATH05374875} showed that it has the generating function:
\begin{equation}\label{eq_genspt}
\sum_{n \ge 1} \spt(n) q^n = \frac{1}{(q; q)_{\infty}} \sum_{n \ge 1} \frac{n q^n}{1 - q^n} + \frac{1}{(q; q)_{\infty}} \sum_{n \ge 1} \frac{(-1)^n q^{\frac{n(3n+1)}{2}} (1 + q^n)}{(1 - q^n)^2}.
\end{equation}
Bringmann \cite[Theorem 1.1]{zbMATH05317179} showed that the function
\begin{equation}\label{eq_modspt}
\mathcal{A}(\tau):=\sum_{n\geq 0}\left(\spt(n)+\frac{24n-1}{12}p(n)\right)q^{n-\frac{1}{24}}-\frac{\sqrt{3}i}{2\pi}\int_{-\overline{\tau}}^{i\infty}\frac{\eta(w)}{(-i(w+\tau))^{3/2}}dw
\end{equation}
is a harmonic weak Maass form, and it satisfies the conditions in Theorem \ref{thm_main3} with $g(\tau)=\eta(\tau)$ and $N=1$, see \cite[Theorem 2.1]{zbMATH06191432}. Using Corollary \ref{coro_type2}, we immediately obtain the congruence of Ono \cite[Theorem 1.1]{zbMATH06133929}:
\[
\spt(\ell^2n-\delta_{\ell})\equiv \left(\frac{3}{\ell}\right)\left(1- \left(\frac{1-24n}{\ell}\right)\right)\left( \spt(n)+\frac{24n-1}{12}p(n)\right)\pmod\ell.
\]
For more examples of Type II and Type III, see Section \ref{section_ex}.

The structure of this paper is as follows. We  prove our three theorems in Section \ref{section_2}, and the proof of Corollary \ref{coro_type2} is also in Section \ref{section_2}. In Section \ref{section_ex}, we provide examples to illustrate our results, and the proofs of Corollary \ref{coro_inconforpar} and Corollary \ref{coro_inconforTypeI} are in Section \ref{section_3.1}.

\section{Proof of results}\label{section_2}
Before beginning the proof, we first recall the basic properties of the Rankin-Cohen bracket. The Rankin-Cohen bracket defined in \eqref{eq_RC_def} satisfies (c.f. \cite{zbMATH03485913})
\[
[f\big|_a\gamma,g\big|_b\gamma]_v=[f,g]_v\big|_{a+b+2v}\gamma,~\gamma\in\operatorname{GL}_2^{+}(\mathbb{R}),
\]
where the slash operator is defined by
\[
(f\big|_k\gamma)(\tau)=(ad-bc)^{k/2}(c\tau+d)f\left(\frac{a\tau+b}{c\tau+d}\right),~\gamma=\begin{pmatrix}
a&b\\
c&d
\end{pmatrix}\in\operatorname{GL}_2^{+}(\mathbb{R}).
\]
In this paper, this property serves two purposes:
\begin{enumerate}
\item If $f$ and $g$ are weakly holomorphic modular forms with weights $a,b$ and multipliers $\nu_1,\nu_2$, respectively, then $[f,g]_v$ is a weakly holomorphic modular form with weight $a+b+2v$ and multiplier $\nu_1\nu_2$.
\item For a cusp $\kappa$ of $\Gamma$, fix $\gamma_{\kappa}\in \SL_2(\mathbb{Z})$ with $\gamma_{\kappa}(i\infty)=\kappa$, we have
\[
[f,g]_v\big|_{a+b+2v}\gamma_{\kappa}=[f\big|_a\gamma_{\kappa},g\big|_b\gamma_{\kappa}]_v.
\]
This shows that we can compute the expansion of $[f,g]_v$ at a given cusp via the expansions of $f$ and $g$ at the corresponding cusps.
\end{enumerate}

\begin{lemma} \label{lemma_RC_valuation}
Suppose that $f \in M_a^!(\Gamma,\nu_1)$ and $g \in M_b^!(\Gamma,\nu_2)$ are weakly holomorphic modular forms. Let $c$ be a cusp of $\Gamma$ and let $\ord_c(\cdot)$ denote the order of vanishing at $c$. Then for any integer $v \ge 0$, the $v$-th Rankin-Cohen bracket satisfies
\[
    \ord_c([f, g]_v) \ge \ord_c(f) + \ord_c(g).
\]
Consequently, if the product $f \cdot g$ is a holomorphic modular form (resp. cusp form), then $[f, g]_v$ is also a holomorphic modular form (resp. cusp form) in $M_{a+b+2v}(\Gamma,\nu_1\nu_2)$.
\end{lemma}
\begin{proof}
By definition, the $v$-th Rankin-Cohen bracket $[f, g]_v$ is a linear combination of the products $D^r f \cdot D^{v-r} g$ for $0 \le r \le v$. Let $q_c$ be the local uniformizing parameter at the cusp $c$. The action of the differential operator $D$ on the local Fourier expansion $\sum a_n q_c^n$ does not decrease the minimum exponent of $q_c$. We deduce that
\begin{align*}
    \ord_c([f, g]_v) &\ge \min_{0 \le r \le v} \ord_c(D^r f \cdot D^{v-r} g) \\
                  &\ge \ord_c(f) + \ord_c(g).
\end{align*}

Furthermore, if $f \cdot g \in M_{k+l}(\Gamma)$, it implies $\ord_c(f \cdot g) = \ord_c(f) + \ord_c(g) \ge 0$ for all cusps $c$. By the established inequality, $\ord_c([f, g]_v) \ge 0$ at all cusps. Since weakly holomorphic forms have no poles in the upper half-plane,  $[f, g]_v$ is holomorphic in the upper half-plane.
\end{proof}

The following lemma is the origin of all our results in this paper.
\begin{lemma}\label{lemma_fgUl}
Let $\ell\geq 3$ be prime and $w$ be an integer with $\ell\nmid sw$ and $v:=(\ell-1)/2$. Let $f(\tau)=\sum_{n\geq 0}a(n)q^{n-\frac{t}{w}}$ with weight $a$ and $g(\tau)=\sum_{m\in\mathbb{Z}}b(m)q^{\frac{sm^2}{w}}$ be a theta series with weight $1/2$ satisfying $b(\ell m)=\chi(\ell)b(m)$, then
\[
\chi(\ell)\left(f(\tau)\cdot g(\ell^2\tau)\right)\big|U_{\ell}\equiv \left(fg-\left(\frac{sw}{\ell}\right)c_a^{-1}[f,g]_{v}\right)\big|U_{\ell} \pmod\ell,
\]
where $c_a= \binom{a-5/2}{\frac{\ell-1}{2}}$ is assumed to be invertible modulo $\ell$.
\end{lemma}
\begin{proof}
We prove this lemma by computing the Fourier series:
\begin{align*}
[f,g]_{v}
&= \sum_{0\leq j\leq v} (-1)^j \binom{a+v-1}{v-j} \binom{v-1/2}{j} \\[1ex]
&\quad \times \sum_{\substack{n\geq 0\\ m\in\mathbb{Z}}} (n-t/w)^j (sm^2/w)^{v-j} a(n)b(m) q^{n+\frac{sm^2-t}{w}} \\[1ex]
&= w^{-v} \sum_{0\leq j\leq v} \binom{a+v-1}{v-j} \binom{v-1/2}{j} \\[1ex]
&\quad \times \sum_{m\in\mathbb{Z}} \sum_{n\geq (sm^2-t)/w} (sm^2-nw)^j s^{v-j} m^{2v-2j} a\left(n+\frac{t-sm^2}{w}\right) b(m)q^{n}.
\end{align*}
Then, applying the operator $U_{\ell}$ to the above equation and reducing modulo $\ell$, we obtain
\begin{equation}\label{eq_rcfg}
[f,g]_{v}\big|U_{\ell}\equiv \left(\frac{sw}{\ell}\right)c_a\cdot\sum_{\substack{m\in\mathbb{Z}\\ \ell\nmid m}}\sum_{n\geq  (sm^2-t)/\ell w}a\left(\ell n+\frac{t-sm^2}{w}\right)b(m)q^n,
\end{equation}
where $c_a=\sum_{0\leq j\leq v}\binom{a+v-1}{v-j}\binom{v-1/2}{j}$. Here we use $w^{\frac{\ell-1}{2}}\equiv \left(\frac{w}{\ell}\right)$; $m^{\ell-1}\equiv 1$ if $\ell\nmid m$ and 0 if $\ell\mid m$. By applying the Vandermonde identity, we can simplify the constant:
\[
c_a=\binom{a+2v-3/2}{v}\equiv \binom{a-5/2}{\frac{\ell-1}{2}}\pmod\ell.
\]
Thus by $\sum_{\ell\mid m}b(m)q^{m^2/w}=\chi(\ell)\sum_{m\in\mathbb{Z}}b(m)q^{\ell^2 m^2/w}$ and \eqref{eq_rcfg}, we obtain
\begin{align*}
\left(fg-\left(\frac{sw}{\ell}\right)c_a^{-1}[f,g]_{v}\right)\big|U_{\ell}&\equiv \sum_{\substack{m\in\mathbb{Z}\\ \ell\mid m}}\sum_{n\geq (sm^2-t)/\ell w}a\left(\ell n+\frac{t-sm^2}{w}\right)b(m)q^{n}\\
&= \chi(\ell)\left(f(\tau)\cdot g(\ell^2\tau)\right)\big|U_{\ell}\pmod\ell,
\end{align*}
the claim follows.
\end{proof}

\begin{proof}[Proof of Theorem \ref{thm_main1}]
From the eta-quotient expression of $f(\tau)$, we can write $f(\tau)g(\ell^2\tau)$ as
\[
f(\tau)g(\ell^2\tau)=\sum_{n\geq 0}a(n)q^{n+\frac{t}{w}(\ell^2-1)}\prod_{\delta\mid N}(q^{\ell^2\delta};q^{\ell^2\delta})^{r_{\delta}}_{\infty}.
\]
Applying the $U_{\ell}$ operator to the above equation and using its property \eqref{eq_propertyUl}, we find that
\begin{equation}\label{eq_fgUl}
(f(\tau)g(\ell^2\tau))\big|U_{\ell}=\prod_{\delta\mid N}(q^{\ell\delta};q^{\ell\delta})^{r_{\delta}}_{\infty}\sum_{n\geq 0}a(\ell n-\tfrac{t}{w}(\ell^2-1))q^{n}.
\end{equation}
Set $a=-1/2$ and $v=(\ell-1)/2$.  By the standard identity for generalized binomial coefficients,
\[
c_{a}=\binom{-3}{v}=(-1)^v\binom{v+2}{2}=(-1)^{(\ell-1)/2}\frac{(\ell+1)(\ell+3)}{8}.
\]
Reducing modulo $\ell$ gives
\begin{equation}\label{eq_valca}
c_{a}\equiv \left(\frac{-1}{\ell}\right)\cdot\frac{3}{8}\pmod \ell.
\end{equation}
By combining the identity established in Lemma \ref{lemma_fgUl} with \eqref{eq_fgUl} and \eqref{eq_valca}, we obtain the congruence in Theorem \ref{thm_main1}. Furthermore, it follows from Lemma \ref{lemma_RC_valuation} that the function $[f,g]_{(\ell-1)/2}$ is a holomorphic modular form in $M_{\ell-1}(\Gamma_0(N))$.
\end{proof}

\begin{proof}[Proof of Theorem \ref{thm_main2}]
The proof is similar to that of Theorem \ref{thm_main1}. Here, we again utilize Lemma \ref{lemma_fgUl} to obtain the desired congruence, and note that
\[
c_a=\binom{-1}{\frac{\ell-1}{2}}= (-1)^{\frac{\ell-1}{2}}\equiv \left(\frac{-1}{\ell}\right)\pmod{\ell}.
\]
\end{proof}

\begin{proof}[Proof of Theorem \ref{thm_main3}]
We use the tool of holomorphic projection. For a basic introduction to holomorphic projection, one may consult \cite[Chapter 10]{zbMATH06828732}. Under our conditions, the function $[F,g]_v$ is a real-analytic modular form and satisfies the conditions required for holomorphic projection. Considering the holomorphic projection of $[F,g]_v$, and applying a theorem of  Mertens \cite[Theorem 5.4]{zbMATH06620622}, we obtain
\[
\pi_{\operatorname{hol}}([F,g]_v)=[f,g]_v+(4w)^{-v}\binom{2v}{v}L_v^{f,g},
\]
and $L_v^{f,g}$ is an $\ell$-integral linear combination of the form 
\[
\begin{split}
\Lambda_{s,u}^{\psi}(\tau;v) = \sum_{r=1}^{\infty} \left( 2 \sum_{\substack{sm^2 - un^2 = r \\ m,n \ge 1}} b(m) \overline{\psi(n)} (\sqrt{s}m - \sqrt{u}n)^{2v+1} \right) q^{r/w} \\
+ \, \overline{\psi(0)} \sum_{m=1}^{\infty} b(m) (\sqrt{s}m)^{2v+1} q^{sm^2/w},
\end{split}
\]
independent of $v$, where the $\psi$ are even characters. When $v=0$, this holomorphic projection yields a holomorphic quasi-modular form of weight 2 and level $N$, and for $v>1$, the holomorphic projection yields a holomorphic modular form in $M_{2v+2}(\Gamma_0(N))$.

Now set $v=(\ell-1)/2$. We find that
\[
\begin{split}
\Lambda_{s,u}^{\psi}(\tau;v) \equiv \sum_{r=1}^{\infty} \left( 2 \sum_{\substack{sm^2 - un^2 = r \\ m,n \ge 1}} b(m) \overline{\psi(n)} \left(\left(\frac{s}{\ell}\right)\sqrt{s}m-\left(\frac{u}{\ell}\right)\sqrt{u}n\right) \right) q^{r/w} \\
+ \, \overline{\psi(0)} \sum_{m=1}^{\infty} b(m) \left(\frac{s}{\ell}\right)\sqrt{s}m q^{sm^2/w}\pmod{\ell}
\end{split},
\]
Therefore, under the condition $\left(\frac{us}{\ell}\right)=1$, we have 
\[
\left(\frac{s}{\ell}\right)\Lambda_{s,u}^{\psi}(\tau;v)\equiv \Lambda_{s,u}^{\psi}(\tau;0)\pmod\ell,
\]
and thus
\[
fg-\left(\frac{-ws}{\ell}\right)[f,g]_v\equiv  \pi_{\operatorname{hol}}(Fg)-\left(\frac{-ws}{\ell}\right) \pi_{\operatorname{hol}}([F,g]_v)\pmod\ell.
\]

The function $\pi_{\operatorname{hol}}(Fg)$ is a quasi-modular form of weight 2. By the structure theorem of quasi-modular forms (c.f. \cite[Proposition 20]{zbMATH05808162}), it is a linear combination of a classical modular form and the Eisenstein series $E_2$. Since $E_2\equiv E_{\ell+1}\pmod\ell$, $\pi_{\operatorname{hol}}(Fg)$ is congruent modulo $\ell$ to a holomorphic modular form of weight $\ell+1$. Thus, the function 
\[
\pi_{\operatorname{hol}}(Fg)-\left(\frac{-ws}{\ell}\right) \pi_{\operatorname{hol}}([F,g]_v)
\]
is congruent to a modular form in $M_{\ell+1}(\Gamma_0(N))$. The subsequent proof procedure is entirely analogous to the proofs of Theorem \ref{thm_main1} and Theorem \ref{thm_main2}.

\end{proof}
Before proving Corollary \ref{coro_type2}, we establish the following property regarding the $U_{\ell}$ operator. The result in the space of modular forms for $N=1$ were given by Serre \cite[Th{\'e}or{\`e}me 11]{serre1973formes}, and our proof essentially follows that of Serre.
\begin{proposition}\label{prop_ul^2}
Let $N$ be a positive integer such that $g(X_0(N)) = 0$ with $N\neq 25$, and let $\ell \ge 3$ be a prime with $(N, \ell) = 1$. For any $f \in M_{\ell+1}(\Gamma_0(N), \mathbb{Z}_{(\ell)})$ whose Fourier expansion at the cusp $i\infty$ has coefficients in  $\mathbb{Z}_{(\ell)}$, we have
\[
    f | U_\ell^2 \equiv f \pmod \ell.
\]
\end{proposition}
\begin{proof}
First we assume that $\ell\geq 5$. We will prove $U_{\ell}^2\equiv\mathrm{Id}\pmod{\ell}$ separately for the space of Eisenstein series and the space of cusp forms. At the end of the proof, we will explain why this is sufficient to prove that $U_{\ell}^2\equiv\mathrm{Id}\pmod{\ell}$ holds in the space of modular forms. 

Let $S_{\ell+1}(\Gamma_0(N), \mathbb{F}_{\ell}):=S_{\ell+1}(\Gamma_0(N),  \mathbb{Z}_{(\ell)})\otimes  \mathbb{F}_{\ell}$ be the space of $q$-series of the cusp forms modulo $\ell$. For $f\in S_{2}(\Gamma_0(N\ell),  \mathbb{Z}_{(\ell)})$, we claim that the trace map (c.f. \cite[Lemma 17]{zbMATH03283550})  
\begin{equation}
f\mapsto f\cdot E_{\ell-1}^{*}+\ell^{\frac{1-\ell}{2}}(f\cdot E_{\ell-1}^{*})\big|_{\ell+1}W^{N\ell}_{\ell}U_{\ell}
\end{equation}
gives an well-defined injection:
\begin{equation}\label{eq_iso}
S_2(\Gamma_0(N\ell), \mathbb{F}_{\ell})\mapsto S_{\ell+1}(\Gamma_0(N), \mathbb{F}_{\ell}),
\end{equation}
where $W^{N}_{\ell}$ is the Atkin-Lehner operator, and
\[
E_{\ell-1}^{*}(\tau):=E_{\ell-1}(\tau)-\ell^{\ell-1}E_{\ell-1}(\ell\tau).
\]
Since the genus of $X_0(N)$ is zero, the space of old forms $S_2(\Gamma_0(N))$ is trivial, and thus every form in the space $S_2(\Gamma_0(N\ell))$ is a $\ell$-new form, which satisfies
\[
f\big|_2W^{N\ell}_{\ell}=-f\big|U_{\ell},
\]
which means that the $\ell$-adic valuation $v_{\ell}(f\big|_{2}W^{N\ell}_{\ell})$ is at least 0. It is not hard to see that $E_{\ell-1}^{*}(\tau)\equiv 1\pmod\ell$ and $v_{\ell}(E_{\ell-1}^{*}\big|_{\ell-1}W^{N\ell}_{\ell})\geq (\ell+1)/2$. It follows from these facts that 
\[
f\cdot E_{\ell-1}^{*}+\ell^{\frac{1-\ell}{2}}(f\cdot E_{\ell-1}^{*})\big|_{\ell+1}W^{N\ell}_{\ell}U_{\ell}\equiv f\pmod\ell,
\] 
and the claim is proved. If the genus of $X_0(N)$ is zero, from the dimension formula (c.f.\cite[Theorem 3.5.1]{diamond2005first}), we have 
\[
\dim_{\mathbb{C}} S_{\ell+1}(\Gamma_0(N)) = \dim_{\mathbb{C}} S_2(\Gamma_0(N\ell)), 
\]
and thus the map \eqref{eq_iso} is indeed an isomorphism. 

The Atkin-Lehner involution $W_\ell^{N\ell}$ acts as a involution on $S_2(\Gamma_0(N\ell))$. Consequently, $U_\ell^2 = W_\ell^2 = \mathrm{id}$ on $S_2(\Gamma_0(N\ell))$ over $\mathbb{C}$. Pulling this back through the mod $\ell$ isomorphism \eqref{eq_iso} yields $U_\ell^2 \equiv \mathrm{Id} \pmod\ell$ on $S_{\ell+1}(\Gamma_0(N),\mathbb{F}_{\ell})$.

The basis for the Eisenstein subspace $\mathcal{E}_{\ell+1}(\Gamma_0(N))$ with trivial Nebentypus consists of Eisenstein series $E_{\ell+1}^{\chi, \bar{\chi}}(d\tau)$, where $d \mid N$, and $\chi$ is a primitive Dirichlet character such that $\text{cond}(\chi)^2 \mid N$. Under our conditions, no congruences occur within the Eisenstein space, and therefore the operator $U_\ell^2 \equiv \mathrm{id} \pmod \ell$ holds for the  space $\mathcal{E}_{\ell+1}(\Gamma_0(N))$ if and only if $\chi(\ell)^2 = 1$ for all permissible characters $\chi$.

It is a classical result that $g(X_0(N)) = 0$ if and only if 
\begin{equation*}
    N \in \{1, 2, 3, 4, 5, 6, 7, 8, 9, 10, 12, 13, 16, 18, 25\}.
\end{equation*}
The condition $\text{cond}(\chi)^2 \mid N$ restricts the existence of $\chi$:
\begin{enumerate}
    \item For $N \in \{1, 2, 3, 4, 5, 6, 7, 8, 10, 12, 13\}$, no non-trivial primitive character satisfies the condition.
    \item For $N \in \{9, 18\}$, $\text{cond}(\chi) = 3$. The  primitive character is the Legendre symbol $(\frac{\cdot}{3})$, which is real. 
    \item For $N = 16$, $\text{cond}(\chi) = 4$. The primitive character is $(\frac{-4}{\cdot})$, which is real. 
    \item For $N = 25$, the primitive characters modulo 5 are of order 4, taking values in $\{\pm 1, \pm i\}$. This is the reason we exclude $N=25$. 
\end{enumerate}

If the direct sum decomposition
\[
M_{\ell+1}(\Gamma_0(N),\mathbb{Z}_{(\ell)})=\mathcal{E}_{\ell+1}(\Gamma_0(N),\mathbb{Z}_{(\ell)})\oplus S_{\ell+1}(\Gamma_0(N),\mathbb{Z}_{(\ell)})
\] 
holds, then according to the above assertion, we can conclude that $U_{\ell}^2\equiv \textrm{Id}\pmod{\ell}$. However, this splitting does not hold in general over $\mathbb{Z}_{(\ell)}$. The failure of this splitting is equivalent to the existence of a congruence between a cusp form and an Eisenstein series modulo $\ell$. Via the Hecke-equivariant isomorphism \eqref{eq_iso}, any such congruence in weight $\ell+1$ yields an Eisenstein congruence in $S_2(\Gamma_0(N\ell),\mathbb{Z}_{(\ell)})$. In the trivial-character case, Mazur’s Eisenstein ideal theory \cite{zbMATH03611460} implies that such a congruence forces  the rational torsion subgroup  $J_0(N\ell)(\mathbb{Q})_{\mathrm{tors}}$ to contain nontrivial $\ell$-primary torsion. For the specific levels under consideration, Yoo’s results \cite{zbMATH07691795} identify the relevant $\ell$-primary torsion subgroup with the $\ell$-primary component of the rational cuspidal divisor class group of $X_0(N\ell)$. Thus, such an Eisenstein congruence can occur only if the rational cuspidal divisor class group of $X_0(N\ell)$ has nontrivial $\ell$-primary torsion. In the cases under consideration, this $\ell$-primary part is trivial with the single exception of $(N, \ell) = (13, 7)$, which is verified by an explicit check using Yoo’s description of the rational cuspidal divisor class group \cite{zbMATH07595062}. Finally, when $N\in\{9,16,18\}$, any Eisenstein congruence associated with a nontrivial quadratic character $\chi$ is algebraically obstructed: the associated generalized Bernoulli numbers, together with the relevant level-raising Euler factors, are divisible only by primes $2$ or $3$. Hence no prime $\ell\ge 5$ can occur. 

For the exceptional cases ($\ell=3$ and $(N,\ell)=(13,7)$), we can directly verify the validity of the congruences via explicit computation. 
\end{proof}

\begin{proof}[Proof of Corollary \ref{coro_type2}]
By applying the operator $U_{\ell}$ to the congruence in Theorem \ref{thm_main2} or Theorem \ref{thm_main3}, and noting that $U_{\ell}$ is an involution modulo $\ell$ under our conditions (Proposition \ref{prop_ul^2}), we obtain
\begin{equation}\label{eq_conl^2}
\chi(\ell)\cdot\prod_{\delta\mid N}(q^{\delta};q^{\delta})^{r_{\delta}}_{\infty}\sum_{n\geq 0}a(\ell^2 n-\tfrac{t}{w}(\ell^2-1))q^n\equiv fg-\left(\frac{-sw}{\ell}\right)\cdot[f(\tau),g(\tau)]_{(\ell-1)/2}\pmod{\ell} .
\end{equation}
Now returning to the explicit expression of $[f,g]_v$, $v=\frac{\ell-1}{2}$. We observe that $v + 1/2 = \ell/2 \equiv 0 \pmod \ell$. Thus, the first binomial coefficient reduces to $\binom{0}{v-j} \pmod l$, which vanishes for all $j < v$. Thus we obtain
\begin{align*}
[f,g]_v&\equiv (-w)^{-v}\sum_{\substack{n\geq 0\\ m\in\mathbb{Z}}}(t-wn)^va(n)b(m)q^{n+\frac{m^2-t}{w}}\\
&\equiv \left(\frac{-w}{\ell}\right) \cdot g \sum_{n\geq 0} \left(\frac{t-wn}{\ell}\right)a(n)q^{n-\frac{t}{w}}.
\end{align*}
Substituting the above congruence into \eqref{eq_conl^2} and canceling $g$ from both sides, the proof of the first assertion is completed by comparing the power series on both sides.

For the second assertion, we follow the method used by Ahlgren and Kim in \cite{zbMATH06865877}. From the above proof and the fact that $U_{\ell}$ is an involution modulo $\ell$,  the congruence 
\[
\sum_{n\geq 0}a(\ell n-\tfrac{t}{w}(\ell^2-1))q^n \equiv 0 \pmod \ell
\]
implies that for each integer $n \ge 0$, we have
\begin{equation}
    \left(1-\left(\frac{s(t-wn)}{\ell}\right)\right)a(n) \equiv 0 \pmod \ell.
\end{equation}
Fix an integer $N > 0$. We construct a sequence of $N$ integers $n_j$ and corresponding distinct primes $p_j$ as follows: If $t = 0$, then $w = 1$. We choose $N$ distinct odd primes $p_j \nmid s$, and set $n_j := p_j$. If $t \neq 0$, by Dirichlet's theorem on arithmetic progressions, we choose $N$ distinct primes $p_j \equiv -t \pmod w$ with $p_j \nmid s$, and set $n_j := (p_j + t)/w$.

In both cases, substitution yields $s(t-wn_j) = -s p_j$. Let $E_N$ be the finite set of primes dividing the non-zero integer $\prod_{j=1}^N a(n_j)$. For any prime $\ell \in S \setminus E_N$, we must have $\left(\frac{-s p_j}{\ell}\right) = 1$, which is equivalent to
\begin{equation}\label{eq_primecon}
    \left(\frac{p_j}{\ell}\right) = \left(\frac{-s}{\ell}\right) \quad \text{for all } 1 \le j \le N.
\end{equation}
Since $p_1, \dots, p_N$ are distinct primes, these $N$ quadratic residue conditions are mutually independent. By  the Chebotarev density theorem, the number of primes $\ell \le X$ satisfying all $N$ conditions is asymptotic to $2^{-N} X / \log X$. 

Any prime $\ell \in S$ either belongs to the finite set $E_N$ or satisfies the conditions in \eqref{eq_primecon}. Thus, the upper density of $S$ is bounded by
\[
    \limsup_{X \to \infty} \frac{\#\{\ell \le X : \ell \in S\}}{X / \log X} \le \limsup_{X \to \infty} \frac{\# E_N}{X / \log X} + 2^{-N} = 2^{-N}.
\]
Since $N$ can be chosen arbitrarily large, letting $N \to \infty$ yields
\[
    \limsup_{X \to \infty} \frac{\#\{\ell \le X : \ell \in S\}}{X / \log X} \le 0.
\]
\end{proof}

\section{Examples}\label{section_ex}
\subsection{Examples of Type I}\label{section_3.1}
We establish the following lemma to prove Corollary \ref{coro_inconforpar}.
\begin{lemma}\label{lemma_valcusp}
For a prime $\ell \ge 5$, let $v = (\ell-1)/2$ and $\alpha\in \mathbb{Z}_{(\ell)}$. Let $f(q) = q^\alpha (a_0 + a_1 q + \dots)$ and $g(q) = 1/f(q)$ be formal power series of weights $a=-1/2$ and $b=1/2$, respectively.  
If $a_0 a_1 \not\equiv 0 \pmod \ell$, the $q^0$ term of $[f, 1/f]_v$ modulo $\ell$ is given by
\[
\binom{2v-2}{v} \alpha^v\equiv \left(\frac{-\alpha}{\ell}\right) \frac{3}{8} \pmod \ell.
\]
The $q^1$ term of $[f, 1/f]_v$ modulo $\ell$ is given by (up to a constant $a_1/a_0$)
\begin{equation}\label{eq_C_v1}
(\alpha+1)^2 (\alpha+1)^v - \left(2\alpha^2 + \frac{3}{2}\alpha\right)\alpha^v + \left(\alpha^2 - \frac{1}{2}\alpha - \frac{1}{2}\right)(\alpha-1)^v \pmod \ell.
\end{equation}
\end{lemma}
\begin{proof}
It is clear that 
\[
\D^j f\cdot \D^{v-j}g=(-1)^j\alpha^v+O(q),
\] 
The first assertion can be obtained immediately using calculations similar to those in the proofs of Lemma \ref{lemma_fgUl} and Theorem \ref{thm_main1}. 

 Since $g(q) = q^{-\alpha}(a_0^{-1} - a_1 a_0^{-2} q + \dots)$, extracting the $q^1$-coefficient of the product $\D^r f \cdot \D^{v-r} g$ yields
\[
\left( \D^r f \cdot \D^{v-r} g \right)_{q^1} = \frac{a_1}{a_0} \left[ (\alpha+1)^r (-\alpha)^{v-r} - \alpha^r (1-\alpha)^{v-r} \right].
\]
By the definition of the Rankin-Cohen bracket, the $q^1$ term  is given by (up to a constant $a_1/a_0$)
\[
P_v(\alpha):= \sum_{r=0}^v (-1)^r \binom{v-3/2}{v-r} \binom{v-1/2}{r} \left[ (\alpha+1)^r (-\alpha)^{v-r} - \alpha^r (1-\alpha)^{v-r} \right].
\]
Using $v \equiv -1/2 \pmod \ell$, we have $\binom{v-1/2}{r} \equiv (-1)^r$ and $\binom{v-3/2}{v-r} \equiv (-1)^{v-r}(v-r+1)$. Thus we obtain
\[
P_v(\alpha) \equiv S(\alpha+1, \alpha) - S(\alpha, \alpha-1) \pmod \ell,
\]
where $S(X,Y) = \sum_{r=0}^v (v-r+1) X^r Y^{v-r}$. Utilizing the standard  identity $(Y-X)^2 S(X,Y) = X^{v+2} - (v+2)XY^{v+1} + (v+1)Y^{v+2}$ and reducing the coefficients modulo $\ell$ (noting $Y-X = -1$), we obtain the closed form claimed in the lemma.
\end{proof} 

\begin{proof}[Proof of Corollary \ref{coro_inconforpar}]
We proceed by contradiction. From the discussion preceding Corollary \ref{coro_inconforpar}, we know that $f_{\ell}\equiv 0\pmod\ell$, which implies that the  $q^1$ term of its Fourier expansion is 0 modulo $\ell$. Substituting $\alpha = -1/24$ in \eqref{eq_C_v1}, we get
\[
P_v\left(-\frac{1}{24}\right) \equiv \frac{1}{576} \left[ 529 \left(\frac{23}{24}\right)^v + 34 \left(-\frac{1}{24}\right)^v - 275 \left(-\frac{25}{24}\right)^v \right] \pmod \ell.
\]
For $\ell \neq 5$, Fermat's Little Theorem implies $(-25/24)^v \equiv (-1)^v (1/24)^v \pmod \ell$. Factoring out $(1/24)^v$, the condition $P_v(-1/24) \equiv 0 \pmod \ell$ is  equivalent to the vanishing of the numerator:
\[
529 \cdot 23^v - 241 \cdot (-1)^v \equiv 0 \pmod \ell.
\]
Let $\epsilon_1 = 23^v \in \{\pm 1\}$ and $\epsilon_2 = (-1)^v \in \{\pm 1\}$. The possible absolute values of $529 \epsilon_1 - 241 \epsilon_2$ over $\mathbb{Z}$ are limited to:
\begin{align*}
|529 - 241| &= 288 = 2^5 \times 3^2, \\
|529 + 241| &= 770 = 2 \times 5 \times 7 \times 11.
\end{align*}
Consequently, the only primes $\ell \ge 5$ that can divide this expression are $5, 7,$ and $11$, and the proof is completed. It is worth noting that the constants 288 and 770 are precisely the two constants that appear in \cite{zbMATH02001026}.
\end{proof}

Before formally proving Corollary \ref{coro_inconforTypeI}, let us examine two typical examples.

An overpartition of $n$ is a non-increasing sequence of natural numbers whose sum is $n$ in which the first occurrence of a number may be overlined. We denote the number of overpartitions of $n$ by $\overline{p}(n)$. We have the generating function
\[
\sum_{n\geq 0}\overline{p}(n)q^n=\frac{(-q;q)_{\infty}}{(q;q)_{\infty}}=\frac{\eta(2\tau)}{\eta(\tau)^2}.
\]
Our theorem implies that 
\begin{corollary}\label{coro_overp}
For every prime $\ell\geq 5$, there is a modular form $h\in M_{\ell-1}(\Gamma_0(2))$ such that
\[
\frac{(q^{\ell};q^{\ell})_{\infty}^2}{(q^{2\ell};q^{2\ell})_{\infty}}\cdot \sum_{n\geq 0}\overline{p}(\ell n)q^n-1\equiv  h \big|U_{\ell}\pmod\ell,
\]
and 
\[
\sum_{n\geq 0}\overline{p}(\ell n)q^n\not\equiv 1\pmod{\ell}.
\]
\end{corollary}
\begin{proof}
The first assertion follows immediately from Theorem \ref{thm_main1}. 
For the second assertion, we employ a proof by contradiction. This implies that
\[
\theta_1(\tau)-1 \equiv h \big| U_{\ell^2} \pmod\ell
\]
for some $h \in M_{\ell-1}(\Gamma_0(2))$, where $\theta_1(\tau) = \sum_{n \in \mathbb{Z}} (-1)^n q^{n^2}$. However, such a congruence between an integral weight modular form and a lacunary theta function cannot hold. 

Let $f(\tau) = \theta_1(\tau) - 1$. For any odd prime $p$, the $U_p$ operator acts on $f(\tau)$ by $f(\tau)\big|U_p = f(p\tau)$. By Dirichlet's theorem, we can choose a prime $p \nmid 2\ell $ such that $1+p^{-1} \not\equiv 0 \pmod\ell$. Modulo $\ell$, the Hecke operator $T_{p} = U_p + p^{\ell-2} V_p$ yields $f(\tau)\big|T_p \equiv (1+p^{-1})f(p\tau)$. Since $1+p^{-1} \not\equiv 0 \pmod\ell$, the Hecke orbit of $f(\tau)$ under $T_p$ spans the infinite sequence of forms $\{f(p^r\tau)\}_{r\geq 0}$. The strictly increasing order of vanishing at $\infty$ forces these forms to be linearly independent over $\mathbb{F}_\ell$, contradicting the finite dimensionality of $M_{\ell-1}(\Gamma_0(2), \mathbb{F}_{\ell})$.
\end{proof}
Note that in this case, we have not utilized any information about $h$. However, since the values of $h$ at each cusp are computable by Lemma \ref{lemma_valcusp}, we can determine its Eisenstein part using the method of undetermined coefficients. The Eisenstein part of the Rankin-Cohen bracket $[\eta(2\tau)/\eta(\tau)^2,\eta(\tau)^2/\eta(2\tau)]$ can be determined explicitly:
\[
\binom{\ell-3}{(\ell-1)/2}\frac{(-1)^{(\ell-1)/2}}{2^{(\ell-1)/2}(2^{\ell-1}-1)}\left(E_{\ell-1}(\tau)-E_{\ell-1}(2\tau)\right).
\]
We can use this to obtain some explicit congruences. Take $\ell=5$ as an example, where the cusp form space is zero, thus $h$ equals to its Eisenstein part. We note that this Eisenstein series can be expressed as an eta-quotient:
\[
E_4(\tau)+h(\tau)=\frac{\eta(\tau)^{16}}{\eta(2\tau)^8},
\]
thus we obtain
\[
\sum_{n\geq 0}\overline{p}(5 n)q^n\equiv \frac{\eta(\tau)^6}{\eta(2\tau)^3}\pmod 5,
\]
which is the congruence given by Treneer \cite[(5.14)]{zbMATH05057203}.

In handling the case where $t\neq 0$, our explicit construction plays a crucial role. We require the following two results concerning modular forms.
\begin{proposition}[$q$-expansion principle (Katz \cite{zbMATH03425708})] Let $f \in M_k(\Gamma_0(N))$ be a modular form whose $q$-expansion at $i\infty$ has coefficients in
$\mathbb{Z}_{(\ell)}$, with $\ell \nmid N$. If the $q$-expansion of $f$ at $i\infty$ is congruent to 0 modulo $\ell$, then for every cusp $\kappa$, the $q$-expansion of $f$ at $\kappa$ is also congruent to 0 modulo any prime $\lambda\mid \ell$ in the coefficient field of that expansion.
\end{proposition}

The following result provides the formulas for the values of a modular form under the action of Hecke operators at various cusps, which is a special case of a result due to Martin \cite[Proposition 4]{zbMATH00963239}.
\begin{proposition}
Let $f \in M_k(\Gamma_0(N))$ be a modular form, and let $\ell$ be a prime such that $\ell \nmid N$ and $\ell^{*}$ be a prime such that $\ell \ell^{*}\equiv 1\pmod{N}$. Let $a_0\left(f, \frac{a}{c}\right)$ denote the value of the modular form at the cusp $\frac{a}{c}$. We have
\begin{equation}\label{eq_hecek}
a_0\left(f | T_\ell, \frac{1}{c}\right) = a_0\left(f, \frac{1}{\ell c}\right) + \ell^{k-1} a_0\left(f, \frac{1}{\ell^{*} c}\right).
\end{equation}
\end{proposition}

Let $\operatorname{pod}(n)$ be the number of partitions of $n$ without repeated odd parts. The generating function of this partition function is a weakly holomorphic modular form of weight $-1/2$:
\[
\sum_{n\geq 0}\operatorname{pod}(n)q^{n-\frac{1}{8}}=\frac{\eta(2\tau)}{\eta(\tau)\eta(4\tau)}.
\]
The reciprocal of this function appears in Mersmann's list. By applying Theorem \ref{thm_main1}, we obtain
\begin{corollary}
For every prime $\ell\geq 5$, there is a modular form $h\in M_{\ell-1}(\Gamma_0(4))$ such that
\[
\frac{(q^{\ell};q^{\ell})_{\infty}(q^{4\ell};q^{4\ell})_{\infty}}{(q^{2\ell};q^{2\ell})_{\infty}}\sum_{n\geq 0}\operatorname{pod}\left(\ell n-\frac{\ell^2-1}{8}\right)q^n\equiv h\big|U_{\ell}\pmod{\ell}.
\]
and
\[
\sum_{n\geq 0}\operatorname{pod}\left(\ell n-\frac{\ell^2-1}{8}\right)q^n\not\equiv0\pmod{\ell}.
\]
\end{corollary}
\begin{proof}
We shall only prove the second assertion. The modular form $h$ has the explicit expression
\[
E_{\ell-1}(\tau)-\binom{\ell-3}{(\ell-1)/2}^{-1}(-8)^{\frac{\ell-1}{2}}\cdot[\eta(2\tau)/(\eta(\tau)\eta(4\tau)),\eta(\tau)\eta(4\tau))/\eta(2\tau)]_{(\ell-1)/2}.
\]
Since the Fourier coefficients of $h$ are $\ell$-integral, the Hecke operators $T_{\ell}$ and $U_{\ell}$ coincide modulo $\ell$. Considering the value of $h$ at the cusp $1/2$ modulo $\ell$, by using \eqref{eq_hecek}, we can obtain the following modulo $\ell$:
\[
a_0(h\big|U_{\ell},1/2)=a_0(h\big|T_{\ell},1/2)=a_0(h,1/2 \ell)=a_0(h,1/2).
\]
The value of $E_{\ell-1}(\tau)$ at the cusp $1/2$ is $1$, so we only need to calculate the value of the Rankin-Cohen bracket at the cusp $1/2$. By Lemma \ref{lemma_valcusp}, this is equivalent to calculating the invariant order of $\eta(2\tau)/(\eta(\tau)\eta(4\tau))$ with respect to the cusp $1/2$. For an eta-quotient $\prod_{\delta\mid N}\eta(\delta\tau)^{r_{\delta}}$, a formula by Ligozat \cite{zbMATH03503412} gives
\begin{equation}\label{eq_Ligozat}
\ord_{a/c}(f)=\frac{N}{24c\gcd(c,N/c)}\sum_{d\mid N}\frac{r_d\gcd(d,c)^2}{d}.
\end{equation}
Using this formula, we can verify that $\ord_{1/2}(\eta(2\tau)/(\eta(\tau)\eta(4\tau)))=0$, thus the value of $h$ at the cusp $1/2$ is $1$ modulo $\ell$. By the $q$-expansion principle, we conclude that $h\big|U_{\ell}$ is not 0 modulo $\ell$ at $i\infty$.
\end{proof}
The Eisenstein part of the modular form $[\eta(2\tau)/(\eta(\tau)\eta(4\tau)),\eta(\tau)\eta(4\tau))/\eta(2\tau)]_{(\ell-1)/2}$ can  be given explicitly:
\[
\binom{\ell-3}{(\ell-1)/2}\frac{(-1/8)^{(\ell-1)/2}}{2^{\ell-1}-1}\left(2^{\ell-1}E_{\ell-1}(\tau)-(2^{\ell-1}+1)E_{\ell-1}(2\tau)+2^{\ell-1}E_{\ell-1}(4\tau)\right).
\]
When $\ell=5$, the space of cusp form $S_4(\Gamma_0(4))$ vanishes. Consequently, the Eisenstein part above is equal to the modular form itself. After some calculations, we find that $h$ itself is an eta-quotient:
\[
h(\tau)=\frac{\eta(\tau)^8\eta(4\tau)^8}{\eta(2\tau)^8}.
\]
Thus we can obtain the following congruence:
\[
\sum_{n\geq 0}\operatorname{pod}\left(5 n-3\right)q^{n-5/8}\equiv \frac{\eta(\tau)^3\eta(4\tau)^3}{\eta(2\tau)^3}\pmod 5.
\]

Now, we follow the proof method from the two examples above to provide the proof of Corollary \ref{coro_inconforTypeI}.
\begin{proof}[Proof of Corollary \ref{coro_inconforTypeI}]
The example of $1/\eta$ has already been discussed in Corollary \ref{coro_inconforpar}, so we exclude it here. We divide the $q$-series in Type I into two categories: the first category consists of those with $t=0$, i.e. the following four:
\[
 \frac{\eta(2\tau)}{\eta(\tau)^2},~\frac{\eta(\tau)^2\eta(4\tau)^2}{\eta(2\tau)^5},~
 \frac{\eta(\tau)\eta(6\tau)}{\eta(2\tau)^2\eta(3\tau)},~
 \frac{\eta(2\tau)\eta(3\tau)\eta(12\tau)}{\eta(\tau)\eta(4\tau)\eta(6\tau)^2},
\]
and the second category comprises those with $t\neq 0$.

The proof of the first case is entirely similar to that in Corollary \ref{coro_overp}, so we omit it. It the second category, we turn to prove that 
\[
\prod_{\delta\mid N}(q^{\ell\delta};q^{\ell\delta})^{r_{\delta}}_{\infty}\sum_{n\geq 0}a(\ell n-\tfrac{t}{w}(\ell^2-1))q^n\not\equiv 0\pmod{\ell},
\]
which is
\[
E_{\ell-1}(\tau)-\left(\frac{-w}{\ell}\right)\frac{8}{3}\cdot[f(\tau),g(\tau)]_{(\ell-1)/2}\big|U_{\ell}\not\equiv 0\pmod{\ell}.
\]
By Lemma \ref{lemma_valcusp} and \eqref{eq_hecek}, the value of the modular form on the left-hand side of the above equation at the cusp $1/c$ is congruent to
\begin{equation}\label{eq_cuspII}
1-\left(\frac{w_{1/c\ell}\ord_{1/c\ell}f/\ord_{i\infty}f}{\ell}\right)\pmod\ell,
\end{equation}
where $w_{1/c\ell}$ is the width of the cusp $1/c\ell$. The order of the eta-quotient $f$ can be calculated using \eqref{eq_Ligozat}. Through concrete calculations, we have discovered that, with the exception of $\eta(\tau)\eta(4\tau)/\eta(2\tau)^3$, there exists a cusp $\kappa$ such that $\ord_{\kappa}(f)=0$, and thus \eqref{eq_cuspII} is not congruent to zero modulo $\ell$ for at least one cusp. The proof is completed by the $q$-expansion principle. 
\end{proof}
\subsection{Examples of Type II}\label{section_ex2}
Let $\spt_{\omega}(n)$ be the smallest parts function  corresponding  to $p_{\omega}(n)$, where $p_{\omega}(n)$ denotes the partitions of $n$ in which the odd parts are less than twice the smallest part. These two functions were first studied by Andrews, Dixit, and Yee \cite{zbMATH06533997}. Wang \cite[(4.3)]{zbMATH06688655} discovered that the generating function of $\spt_{\omega}(n)$, when restricted to odd indices, is a weakly holomorphic modular form:
\[
\sum_{n=0}^{\infty} \spt_{\omega}(2n+1) q^{2n+\frac{11}{12}} = \frac{\eta(4\tau)^8}{\eta(2\tau)^5}.
\]
Note that 
\[
\frac{\eta(4\tau)^8}{\eta(2\tau)^5}\cdot \eta(2\tau)=\frac{\eta(4\tau)^8}{\eta(2\tau)^4}
\]
is a cusp form in $M_2(\Gamma_0(4))$. The pair $(\eta(4\tau)^8/\eta(2\tau)^5,\eta(2\tau))$ satisfies all the conditions in Theorem \ref{thm_main2}, so we have the following result, in which the second congruence is a special case of \cite[Theorem 1.2]{zbMATH08135742}.
\begin{corollary}\label{coro_sptw}
For every prime $\ell\geq 5$, there is a cusp form $h(\tau)\in S_{\ell+1}(\Gamma_0(4))$ such that  
\[
(q^{2\ell};q^{2\ell})_{\infty}\sum_{n\geq 0}\spt_{\omega}\left(\ell(2 n+1)-\frac{\ell^2-1}{12}\right)q^{2n+1}\equiv h(\tau)\big|U_{\ell}\pmod{\ell},
\]
and for every $n\geq 0$,
\[
\spt_{\omega}\left(\ell^2 (2n+1)-\frac{\ell^2-1}{12}\right)\equiv \left(\frac{3}{\ell}\right)\left(1-\left(\frac{-11-24n}{\ell}\right)\right)\spt_{\omega}(2n+1)\pmod\ell.
\]
\end{corollary}

Noting that the congruence in the above result occurs within the space of cusp forms, our explicit construction establishes this fact, provided that we invoke the following result.
\begin{lemma}\label{lem_cuspII}
Let $\ell \ge 3$ be a prime and $v = (\ell-1)/2$. Let $f(q) = a(0) q^{-\alpha} + \cdots$ and $g(q) = b(0) q^\beta + \cdots$ be formal power series of weights $a = 3/2$ and $b = 1/2$, respectively. Assume that $a(0)b(0)\not\equiv0\pmod\ell$.  Then the leading coefficient of the $q$-expansion of  $fg-\left(\frac{-sw}{\ell}\right)[f,g]_v$ is given by
\[
a(0)b(0)\left(1-\left(\frac{sw\alpha}{\ell}\right)\right) \pmod \ell.
  \]
\end{lemma}
\begin{proof}
By definition, the leading Fourier coefficient $[f,g]_v$ is given by
\[
a(0) b(0) \sum_{j=0}^v (-1)^j \binom{v+1/2}{v-j} \binom{v-1/2}{j} (-\alpha)^r \beta^{v-r}. 
\]
The summation modulo $\ell$ collapses to the single term $j=v$; thus, the above is congruent to
\[
  a(0) b(0)  (-1)^v \binom{v+1/2}{0} \binom{v-1/2}{v} (-tw)^v \equiv  a(0) b(0)  (-\alpha)^v \pmod \ell.
\]
Thus the leading coefficient of $fg-\left(\frac{-sw}{\ell}\right)[f,g]_v$ becomes
\[
a(0)b(0) \left( 1 - \left(\frac{sw\alpha}{\ell}\right) \right) \pmod \ell. 
\]
Given $a(0)b(0) \not\equiv 0 \pmod \ell$, the coefficient vanishes if and only if $\left(\frac{sw\alpha}{\ell}\right) = 1$.
\end{proof}

\begin{proof}[Proof of Corollary \ref{coro_sptw}]
It suffices to show that $h$ is a cusp form, as the remaining results follow directly from Theorem \ref{thm_main2} and Corollary \ref{coro_type2}. By Ligozat's formula \eqref{eq_Ligozat}, we see that the invariant orders of the modular form $\frac{\eta(4\tau)^8}{\eta(2\tau)^5}$ with respect to the cusps 0 and $1/2$ are $-1/48$ and $-1/12$, respectively, whereas those of $\eta(2\tau)$ are $1/48$ and $1/12$. It then follows from Lemma \ref{lem_cuspII} that the values of $h$ at both cusps 0 and $1/2$ vanish modulo $\ell$. Consequently, the $q$-expansion principle ensures that $h$ is congruent modulo $\ell$ to a cusp form, which completes the proof.
\end{proof}

Let $\overline{\spt}_{\omega}(n)$ denote the number of smallest parts in the overpartition of $n$ in which the smallest part is always overlined and all odd parts are less than twice 
the smallest part. In the same paper, Wang \cite[Theorem 3.1]{zbMATH06688655} found that
\[
\sum_{n\geq 0}\overline{\spt}_{\omega}(2n+1)q^{2n+1}=\frac{\eta(4\tau)^9}{\eta(2\tau)^6},
\]
note that
\[
\frac{\eta(4\tau)^9}{\eta(2\tau)^6}\cdot \frac{\eta(2\tau)^2}{\eta(4\tau)}=\frac{\eta(4\tau)^8}{\eta(2\tau)^4}
\]
is a modular form in $M_{2}(\Gamma_0(4))$. Taking $g(\tau)=\frac{\eta(2\tau)^2}{\eta(4\tau)}$, in analogy to Corollary \ref{coro_sptw}, we obtain the following.
\begin{corollary}\label{coro_osptw}
For every prime $\ell\geq 3$, there is a cusp form $h(\tau)\in S_{\ell+1}(\Gamma_0(4))$ such that  
\[
\frac{(q^{2\ell};q^{2\ell})_{\infty}^2}{(q^{4\ell};q^{4\ell})_{\infty}}\sum_{n\geq 0}\overline{\spt}_{\omega}\left(2\ell n+\ell\right)q^{2n+1}\equiv h(\tau)\big|U_{\ell}\pmod{\ell},
\]
and for every $n\geq 0$,
\[
\overline{\spt}_{\omega}\left(\ell^2 (2n+1)\right)\equiv\left(1-\left(\frac{-4n-2}{\ell}\right)\right)\overline{\spt}_{\omega}(2n+1)\pmod\ell.
\]
\end{corollary}
The second congruence in it was generalized to prime powers by Mao \cite{zbMATH07567828}.

The generating functions of the two examples mentioned above are, in fact, mock modular forms of weight $3/2$. They all share the same origin: since their shadows are theta series, there always exists some arithmetic progression that makes the mock modular form behave as a weakly holomorphic modular form over this arithmetic progression, for example see \cite{zbMATH06334779} for the cases of mock theta function. As we will see in Section \ref{section_ex3}, considering its corresponding harmonic Maass form directly and examining it within Type III will yield more general congruence results.

Let $\theta(\tau)=\sum_{n\in\mathbb{Z}}q^{n^2}$ be the standard theta function, which is  an eta-quotient of weight $1/2$. Then $\theta(\tau)^3$ is a holomorphic modular form of weight $3/2$ and the pair $(\theta(\tau)^3,\theta(\tau))$ satisfies the conditions in Theorem \ref{thm_main2}. The $n$-th Fourier coefficient of $\theta(\tau)^3$  represents the number of ways to express $n$ as a sum of three squares, denoted by $r_3(n)$. By the results established in this paper, we have the following congruence.
\begin{corollary}\label{coro_theta3}
For every prime $\ell\geq 3$, there is a modular form $h(\tau)\in M_{\ell+1}(\Gamma_0(4))$ such that  
\[
\theta(\ell^2\tau) \sum_{n\geq 0}r_3(\ell n)q^n\equiv h(\tau)\big|U_{\ell}\pmod{\ell},
\]
and for every $n\geq 0$,
\[
r_3(\ell^2n)\equiv\left(1-\left(\frac{-n}{\ell}\right)\right)r_3(n)\pmod\ell.
\]
\end{corollary}
Another explanation for the second congruence above is that $\theta(\tau)^3$ is an eigenvector of the Hecke operators $T_{\ell^2}$:
\[
r_3(\ell^2 n)+ \left(\frac{-n}{\ell}\right)r_3(n)+ \ell r_3(n/\ell^2)=(\ell+1)r_3(n).
\]

\subsection{Examples of Type III}\label{section_ex3}
Andrews, Dixit and Yee \cite[Lemma 6.1]{zbMATH06533997} expressed the generating function of $\spt_{\omega}(n)$ as a Appell-Lerch sum: 
\begin{equation}\label{eq_sptw}
\sum_{n=1}^{\infty} \spt_{\omega}(n) q^n = \frac{1}{(q^2; q^2)_{\infty}} \sum_{n=1}^{\infty} \frac{n q^n}{1 - q^n} + \frac{1}{(q^2; q^2)_{\infty}} \sum_{n=1}^{\infty} \frac{(-1)^n (1 + q^{2n}) q^{n(3n+1)}}{(1 - q^{2n})^2}.
\end{equation}
By comparing it with the generating function of the $\spt(n)$ in \eqref{eq_genspt}, we can see that the generating function of $\spt_{\omega}(n)$ is essentially a mock modular form of weight $3/2$ and level $2$. By taking $g(\tau)=\eta(2\tau)$, the conditions in Theorem \ref{thm_main3} are satisfied. Thus, we have the following congruence which is a strengthened version of the one in Corollary \ref{coro_sptw}.
\begin{corollary}
For every prime $\ell\geq 5$, there is a cusp form $h(\tau)\in S_{\ell+1}(\Gamma_0(2))$ such that  
\[
(q^{2\ell};q^{2\ell})_{\infty}\sum_{n\geq 0}\spt_{\omega}\left(\ell n-\frac{\ell^2-1}{12}\right)q^n\equiv h(\tau)\big|U_{\ell}\pmod{\ell},
\]
and for every $n\geq 0$,
\[
\spt_{\omega}\left(\ell^2 n-\frac{\ell^2-1}{12}\right)\equiv \left(\frac{3}{\ell}\right) \left(1-\left(\frac{1-12n}{\ell}\right)\right)\left(\spt_{\omega}(n)+\frac{12n-1}{24}\cdot p\left(n/2\right)\right)\pmod\ell,
\]
where $p(n/2)=0$ if $n$ is odd.
\end{corollary}
\begin{proof}
Let $S(q):=\sum_{n\geq 0}\spt(n)q^n$ and $S_{\omega}(q):=\sum_{n\geq 0}\spt_{\omega}(n)q^n$  be the generating functions of the two smallest part functions. Note the basic identity 
\[
\sum_{n=1}^\infty \frac{nq^n}{1-q^n} = \frac{1-E_2(\tau)}{24}.
\]
Comparing the generating function \eqref{eq_genspt} and \eqref{eq_sptw}, we observe that the difference between $S_{\omega}(q)$ and $S(q^2)$ arises from the quasimodular part:
\begin{equation}\label{eq_Srelation}
S_{\omega}(q) - S(q^2)= \frac{E_2(2\tau) - E_2(\tau)}{24(q^2;q^2)_\infty}.
\end{equation}
From \eqref{eq_modspt}, the holomorphic part of the form $\mathcal{A}(2\tau)$ is given by
\begin{equation}\label{eq_Ftau}
\mathcal{A}^+(2\tau) = q^{-\frac{1}{12}} S(q^2) - \frac{E_2(2\tau)}{12\eta(2\tau)}.
\end{equation}
Substituting \eqref{eq_Srelation} into \eqref{eq_Ftau}, we deduce that
\begin{equation}\label{eq_Ftauex}
\mathcal{A}^+(2\tau) = q^{-\frac{1}{12}} S_{\omega}(q) + \frac{E_2(\tau) - 3E_2(2\tau)}{24\eta(2\tau)}. 
\end{equation}
Now we consider the harmonic weak Maass form:
\[
\mathcal{A}_{\omega}(\tau):=\mathcal{A}(2\tau)-\frac{E_2(\tau) - 2E_2(2\tau)}{24\eta(2\tau)},
\]
and by \eqref{eq_Ftauex}, the holomorphic part of $\mathcal{A}_{\omega}(\tau)$ is 
\[
q^{-\frac{1}{12}} S_{\omega}(q) - \frac{P(q^2)}{24\eta(2\tau)}=\sum_{n\geq 0}\left(\spt_{\omega}(n) + \frac{12n-1}{24}p\left(n/2\right)\right)q^{n-\frac{1}{12}}.
\]
It is a mock modular form of weight $3/2$, and the pair $(\mathcal{A}_{\omega}(\tau),\eta(2\tau))$ belongs to Type III with $N=2$. Theorem \ref{thm_main3} asserts the existence of the modular form $h$. To prove that it is a cusp form, we must calculate the value of this modular form at the different cusps. Since $X_0(2)$ has only two cusps, it suffices to verify that its value at the cusp 0 is 0. 

Using the spectral-theoretic account of holomorphic projection (see \cite{zbMATH06685253}), it is not hard to see 
\[
\pi_{\hol}(f\big|_{k}\gamma)=\pi_{\hol}(f)\big|_{k}\gamma.
\]
This property indicates that we can utilize the modular transformation properties of $\mathcal{A}_{\omega}(\tau)$ to calculate the values of its holomorphic projection at various cusps. By the modular transformation properties of $\mathcal{A}_{\omega}(\tau)$, we know that the invariant order of $\mathcal{A}^{+}_{\omega}(\tau)$ with respect to the cusp 0 is $-1/48$. Consequently, by Lemma \ref{lem_cuspII}, the $q$-expansion of
\[
\pi_{\hol}(\mathcal{A}_{\omega}(\tau)\eta(2\tau))E_{\ell-1}(\tau)-\pi_{\hol}([\mathcal{A}_{\omega}(\tau),\eta(2\tau)]_{(\ell-1)/2})
\]
takes the value 0 at the cusp 0 modulo $\ell$, and thus is congruent to a cusp form by the $q$-expansion principle.
\end{proof}

For the $\overline{\spt}_{\omega}(n)$, Andrews, Dixit, Schultz and Yee \cite[Theorem 5.1]{zbMATH06828531} showed that its generating function also has a representation in terms of Appell-Lerch sums:
\[
\sum_{n \ge 1} \overline{\spt}_{\omega}(n) q^n = \frac{(-q^2; q^2)_{\infty}}{(q^2; q^2)_{\infty}} \sum_{n \ge 1} \frac{n q^n}{1 - q^n} + 2 \frac{(-q^2; q^2)_{\infty}}{(q^2; q^2)_{\infty}} \sum_{n \ge 1} \frac{(-1)^n q^{2n(n+1)}}{(1 - q^{2n})^2}.
\]
This expression demonstrates that the generating function of the $\overline{\spt}_{\omega}(n)$  is a mock modular form. Consequently, we obtain the following result, which is a strengthened version of Corollary \ref{coro_osptw}.
\begin{corollary}
For every prime $\ell\geq 5$, there is a modular form $h(\tau)\in M_{\ell+1}(\Gamma_0(4))$ such that  
\[
\frac{(q^{2\ell};q^{2\ell})_{\infty}^2}{(q^{4\ell};q^{4\ell})_{\infty}}\sum_{n\geq 0}\left(\overline{\spt}_{\omega}\left(\ell n\right)-\frac18\cdot\overline{p}(\ell n/2)\right) q^n\equiv h(\tau)\big|U_{\ell}\pmod{\ell},
\]
and for every $n\geq 0$,
\[
\overline{\spt}_{\omega}\left(\ell^2 n\right)-\frac18\cdot\overline{p}(\ell^2 n/2)\equiv  \left(1-\left(\frac{-2n}{\ell}\right)\right)\left(\overline{\spt}_{\omega}\left(n\right)-\frac18\cdot\overline{p}(n/2)\right)\pmod\ell,
\]
where $\overline{p}(n/2)=0$ if $n$ is odd.
\end{corollary}
\begin{proof}
It was proved in \cite[Theorem 4.1]{zbMATH05665740} that the function
\begin{equation}\label{eq_Harspto}
\begin{aligned}
\mathcal{M}(\tau)&:=\frac{(-q;q)_{\infty}}{(q;q)_{\infty}}\sum_{n\geq 1}\frac{(-1)^nq^{n(n+1)}}{(1-q^n)^2}-\frac{(-q;q)_{\infty}}{(q;q)_{\infty}}\left(\frac{1}{24}+\frac{1}{12}E_2(2\tau)\right)\\
&+\frac{1}{4\sqrt{2}\pi i}\int_{-\bar{z}}^{i\infty}\frac{\eta(\tau)^2/\eta(2\tau)}{(-i(\tau+z))^{3/2}}d\tau.
\end{aligned}
\end{equation}
is a harmonic weak Maass form of weight $3/2$ on $\Gamma_0(16)$ with the $\theta$ multiplier, and from \cite[Proposition 6]{zbMATH06398837}, the pair $(\mathcal{M}(\tau),\eta^2(\tau)/\eta(2\tau))$ belongs to Type III with $N=2$.   We utilize these facts to establish the exact modularity for the $\overline{\spt}_{\omega}$-function. 

Define $\mathcal{M}^{+}(\tau)$ as the holomorphic part of $\mathcal{M}(\tau)$. Let $\overline{S}_{\omega}(q)$ be the generating function for $\overline{\spt}_{\omega}(n)$. We can rewrite this generating function as
\begin{equation}\label{eq_So_rewrite}
\overline{S}_{\omega}(q) = \frac{\eta(4\tau)}{\eta(2\tau)^2} \frac{1-E_2(\tau)}{24} + 2M(q^2).
\end{equation}
where $M(q)$ denotes the Appell-Lerch sum. Isolating $2M(q^2)$ from \eqref{eq_Harspto} and substituting it into \eqref{eq_So_rewrite} yields
\begin{equation}\label{eq_So_isolated}
\overline{S}_{\omega}(q) = 2\mathcal{M}^+(2\tau) - \frac{3 - E_2(\tau) + 4E_2(4\tau)}{24} \frac{\eta(4\tau)}{\eta(2\tau)^2} . 
\end{equation}
Now we consider the harmonic weak Maass form
\[
\overline{\mathcal{M}}_{\omega}(\tau):=2\mathcal{M}(2\tau)-\frac{E_2(\tau)-4E_2(4\tau)}{24}  \frac{\eta(4\tau)}{\eta(2\tau)^2}, 
\]
by \eqref{eq_So_isolated}, whose holomorphic part is given by 
\[
\overline{S}_{\omega}(q) - \frac{1}{8}\frac{\eta(4\tau)}{\eta(2\tau)^2}=\sum_{n\geq 0} \left(\overline{\spt}_{\omega}(n) - \frac{1}{8}\cdot\overline{p}\left(n/2\right)\right)q^n.
\]
Thus the pair $(\overline{\mathcal{M}}_{\omega}(\tau),\eta(2\tau)^2/\eta(4\tau))$ with $N=4$ belongs to Type III. By Theorem \ref{thm_main3} and Corollary \ref{coro_type2}, we immediately obtain the conclusion in the corollary.
\end{proof}
It is worth noting that the above results yield an interesting congruence relation: 
\[
8\cdot \overline{\spt}_{\omega}\left(2\ell^2 n\right)\equiv \overline{p}(\ell^2 n)\pmod{\ell}, \text{ if }\left(\frac{-n}{\ell}\right)=1.
\]
By applying the method presented in this paper, we can obtain further congruences of this type. We refrain from providing more examples here and turn instead to arithmetic examples.

The Hurwitz class number $H(N)$ is defined by the class number of positive definite binary quadratic forms of discriminant $–N$, where forms are weighted by $2/g$ for $g$ the order of their automorphism group. It has the following connection with the $\theta(\tau)^3$  discussed in Corollary \ref{coro_theta3}, which can be traced back to Gauss:
\[
\theta(\tau)^3=12\sum_{n\geq 0}\left(H(4n)-2H(n)\right)q^n.
\]
A famous result by Zagier \cite{zbMATH03505081} states that the generating function of class numbers $\sum_{n\geq 0}H(n)q^n$ with $H(0):=-1/12$ is a mock modular form of weight $3/2$. The harmonic Maass form corresponding to this generating function is
\[
\mathcal{H}(\tau) = \sum_{n \ge 0} H(n) q^n + \frac{1}{8\pi\sqrt{y}} + \frac{1}{4\sqrt{\pi}} \sum_{n \ge 1} n \Gamma \left( -\frac{1}{2}, 4\pi n^2 y \right) q^{-n^2},
\]
and it is of level 4 with the shadow $\theta(\tau)$. Taking $g(\tau)=\theta(\tau)$, by Theorem \ref{thm_main3} and Corollary \ref{coro_type2}, we have
\begin{corollary}
For every prime $\ell\geq 3$, there is a modular form $h(\tau)\in M_{\ell+1}(\Gamma_0(4))$ such that  
\[
\theta(\ell^2\tau) \sum_{n\geq 0}H(\ell n)q^n\equiv h(\tau)\big|U_{\ell}\pmod{\ell},
\]
and for every $n\geq 0$,
\[
H(\ell^2n)\equiv\left(1-\left(\frac{-n}{\ell}\right)\right)H(n)\pmod\ell.
\]
\end{corollary}
Similar to the case of $r_3(n)$, the origin of the second congruence above lies in the behavior of $H(n)$ with respect to the Hecke operators:
\[
H(\ell^2n)+ \left(\frac{-n}{\ell}\right)H(n) + \ell H(n/\ell^2) =(\ell+1)H(n).
\]

Alfes, Bringmann and Lovejoy proved that the generating function of Hurwitz class numbers defined on the arithmetic progression $8n-1$ is also mock modular form of weight 3/2 \cite{zbMATH05967727}:
\[
\sum_{n\geq 1}H(8n-1)q^{n-\frac{1}{8}}=\frac{\eta(\tau)}{\eta(2\tau)^2}\sum_{n\in\mathbb{Z}}\frac{q^{2n^2+3n+1}}{(1-q^{2n+1})^2},
\]
and its completion 
\[
\mathcal{H}_{8,-1}(\tau):=\frac{\eta(\tau)}{\eta(2\tau)^2}\sum_{n\in\mathbb{Z}}\frac{q^{2n^2+3n+1}}{(1-q^{2n+1})^2}+\frac{1}{2\pi i}\int_{-\bar{z}}^{i\infty}\frac{\eta(\tau)^2/\eta(2\tau)}{(-i(\tau+z))^{3/2}}d\tau
\]
satisfies the conditions in Theorem \ref{thm_main3} with $g(\tau)=\eta(2\tau)^2/\eta(\tau)$ and $N=2$ (see \cite[Proposition 6]{zbMATH06398837}). We obtain the following results concerning the congruence of the Hurwitz class numbers:
\begin{corollary}
For every prime $\ell\geq 3$, there is a modular form $h(\tau)\in M_{\ell+1}(\Gamma_0(2))$ such that  
\[
\frac{(q^{2\ell};q^{2\ell})_{\infty}^2}{(q^{\ell};q^{\ell})_{\infty}} \sum_{n\geq 0}H(8\ell n-\ell^2 )q^n\equiv h(\tau)\big|U_{\ell}\pmod{\ell},
\]
and
\[
\sum_{n\geq 0}H(8\ell n-\ell^2 )q^n\not\equiv 0\pmod{\ell},
\]
and for every $n\geq 0$,
\[
H(\ell^2(8n-1))\equiv\left(1-\left(\frac{1-8n}{\ell}\right)\right)H(8n-1)\pmod\ell.
\]

\end{corollary}
\begin{proof}
The only result in the corollary that cannot be obtained directly is the second non-congruence result. We can prove this by calculating the value of $h\big|U_{\ell}$ at the cusp 0 and utilizing the $q$-expansion principle as in the proof of Corollary \ref{coro_inconforTypeI}. By the proof of \cite[Lemma 6.1]{zbMATH05665740}, it related to the $\mathcal{M}(\tau)$ defined in \eqref{eq_Harspto} by
\[
\mathcal{H}_{8,-1}\left(-1/\tau\right)=-(-i\tau)^{\frac{3}{2}}\mathcal{M}(\tau/2).
\]
By this  transformation property, the value of the holomorphic part of  $\mathcal{H}_{8,-1}(\tau)$ at the cusp 0 is $1/8$. Simultaneously, from the transformation properties of the Dedekind eta function, we know that the value of $g(\tau)$ at the cusp 0 is $1/2$. Therefore, by Lemma \ref{lem_cuspII}, we conclude that $h\big|U_{\ell}$ is non-zero modulo $\ell$ at the cusp 0. This completes the proof.

It is worth mentioning that this specific non-congruence is also consistent with the  results of \cite{zbMATH07936785}, who classified such non-holomorphic congruences under certain regularity conditions. Our approach here provides an independent derivation.
\end{proof}

\appendix
\section{Primitive eta-quotients of weight 1/2}

\begin{table}[htbp]
\centering
\renewcommand{\arraystretch}{2.0}
\caption{The 14 primitive eta-quotients and their explicit $q$-series expansions}
\label{tab:eta_products_qseries}
\begin{tabular}{cccc}
\toprule
$w$ & $\eta$-quotient & $q$-series expansion & $\chi(\ell)~(\ell\geq 5)$ \\
\midrule
$1$ 
& $\dfrac{\eta(2\tau)^5}{\eta(\tau)^2\eta(4\tau)^2}$ 
& $\displaystyle \sum_{n \in \mathbb{Z}} q^{n^2}$ 
& $1$ \\

& $\dfrac{\eta(\tau)^2}{\eta(2\tau)}$ 
& $\displaystyle \sum_{n \in \mathbb{Z}} (-1)^n q^{n^2}$ 
& $1$ \\

& $\dfrac{\eta(2\tau)^2\eta(3\tau)}{\eta(\tau)\eta(6\tau)}$ 
& $\displaystyle \sum_{n \in \mathbb{Z}} \cos\left(\dfrac{n\pi}{3}\right) q^{n^2}$ 
& $1$ \\

& $\dfrac{\eta(\tau)\eta(4\tau)\eta(6\tau)^2}{\eta(2\tau)\eta(3\tau)\eta(12\tau)}$ 
& $\displaystyle \sum_{n \in \mathbb{Z}}\cos\left(\frac{2n\pi}{3}\right) q^{n^2}$ 
& $1$ \\
\midrule

$3$ 
& $\dfrac{\eta(2\tau)^2\eta(3\tau)\eta(12\tau)}{\eta(\tau)\eta(4\tau)\eta(6\tau)}$ 
& $\dfrac{1}{2} \displaystyle \sum_{n \in \mathbb{Z}} \chi_0^{(3)}(n) q^{n^2/3}$ 
& $1$ \\

& $\dfrac{\eta(\tau)\eta(6\tau)^2}{\eta(2\tau)\eta(3\tau)}$ 
& $\dfrac{1}{2} \displaystyle \sum_{n \in \mathbb{Z}} (-1)^{n-1} \chi_0^{(3)}(n) q^{n^2/3}$
& $1$ \\
\midrule

$8$ 
& $\dfrac{\eta(2\tau)^2}{\eta(\tau)}$ 
& $\dfrac{1}{2} \displaystyle \sum_{n \in \mathbb{Z}} \chi_0^{(2)}(n) q^{n^2/8}$ 
& $1$ \\

& $\dfrac{\eta(\tau)\eta(4\tau)}{\eta(2\tau)}$ 
& $\dfrac{1}{2} \displaystyle \sum_{n \in \mathbb{Z}} \left(\dfrac{2}{n}\right) q^{n^2/8}$
& $\left(\dfrac{2}{\ell}\right)$  \\

& $\dfrac{\eta(2\tau)^5\eta(3\tau)\eta(12\tau)}{\eta(\tau)^2\eta(4\tau)^2\eta(6\tau)^2}$ 
& $\dfrac{1}{2} \displaystyle \sum_{n \in \mathbb{Z}} \left(3\chi_0^{(3)}(n)-2\right)\left(\dfrac{2}{n}\right) q^{n^2/8}$ 
& $\left(\dfrac{2}{\ell}\right)$ \\

& $\dfrac{\eta(\tau)^2\eta(6\tau)}{\eta(2\tau)\eta(3\tau)}$ 
& $\dfrac{1}{2} \displaystyle \sum_{n \in \mathbb{Z}} \left(3\chi_0^{(3)}(n)-2\right)\chi_0^{(2)}(n)q^{n^2/8}$ 
& $1$ \\
\midrule

$24$ 
& $\eta(\tau)$ 
& $\dfrac{1}{2} \displaystyle \sum_{n \in \mathbb{Z}} \left(\dfrac{12}{n}\right) q^{n^2/24}$  
& $\left(\dfrac{3}{\ell}\right)$ \\

& $\dfrac{\eta(2\tau)\eta(3\tau)^2}{\eta(\tau)\eta(6\tau)}$ 
& $\dfrac{1}{2} \displaystyle \sum_{n \in \mathbb{Z}} \chi_0^{(12)}(n) q^{n^2/24}$ 
& $1$ \\

& $\dfrac{\eta(2\tau)^3}{\eta(\tau)\eta(4\tau)}$ 
& $ \dfrac{1}{2} \displaystyle \sum_{n \in \mathbb{Z}} \left(\dfrac{6}{n}\right) q^{n^2/24}$ 
& $\left(\dfrac{6}{\ell}\right)$ \\

& $\dfrac{\eta(\tau)\eta(4\tau)\eta(6\tau)^5}{\eta(2\tau)^2\eta(3\tau)^2\eta(12\tau)^2}$ 
& $\dfrac{1}{2} \displaystyle \sum_{n \in \mathbb{Z}} \left(\dfrac{8}{n}\right)\chi_0^{(3)}(n) q^{n^2/24}$ 
& $\left(\dfrac{2}{\ell}\right)$ \\
\bottomrule
\end{tabular}
\vspace{1.5ex}
\parbox{12cm}{\small \textit{Note:} $\chi_0^{(N)}(n)$ denotes the principal Dirichlet character modulo $N$; \\
$\left(\frac{a}{n}\right)$ denotes the Kronecker symbol.}
\end{table}
\clearpage

\providecommand{\bysame}{\leavevmode\hbox to3em{\hrulefill}\thinspace}
\providecommand{\MR}{\relax\ifhmode\unskip\space\fi MR }
\providecommand{\MRhref}[2]{%
  \href{http://www.ams.org/mathscinet-getitem?mr=#1}{#2}
}
\providecommand{\href}[2]{#2}

\end{document}